\documentclass[11pt,letterpaper,reqno]{amsart}
\usepackage{graphicx}
\usepackage{amsmath}
\usepackage{amssymb}
\usepackage{amsfonts}
\usepackage{amsrefs}
\usepackage{amsthm}
\usepackage{cancel}
\usepackage{enumerate}
\usepackage[mathscr]{eucal}

\newcommand{\R}{\mathbb{R}}
\newcommand{\Bcal}{{\mathcal{B}}}

\newcommand{\inv}{^{-1}}
\newcommand{\ol}{\overline}

\newcommand{\sm}{\setminus}

\newtheorem{thm}{Theorem}
\newtheorem{lemma}[thm]{Lemma}
\newtheorem{prop}[thm]{Proposition}

\newtheorem{obs}[thm]{Observation}
\newtheorem{prob}{Problem}

\theoremstyle{definition}
\newtheorem{definition}[thm]{Definition}

\theoremstyle{remark}
\newtheorem*{remark}{Remarks}
\newtheorem*{ack}{Acknowledgements}
\newtheorem*{outline}{Outline}

\renewcommand{\emptyset}{\varnothing}

\DeclareMathOperator{\Hess}{Hess}
\DeclareMathOperator{\Ric}{Ric}
\DeclareMathOperator{\Int}{int}

\begin{document}
\title[Fill-ins of nonnegative scalar curvature]{Fill-ins of nonnegative scalar curvature, static metrics, and quasi-local mass}
%\date{\today}
\author{Jeffrey L. Jauregui}
\address{Dept. of Mathematics,
David Rittenhouse Lab,
209 S. 33rd Street,
Philadelphia, PA 19104}
\email{jauregui@math.upenn.edu}

\begin{abstract}
Consider a triple of ``Bartnik data'' $(\Sigma, \gamma,H)$, where $\Sigma$ is a topological 2-sphere with 
Riemannian metric $\gamma$ and positive function $H$.  We view Bartnik data as a boundary condition for 
the problem of finding a compact Riemannian 3-manifold $(\Omega,g)$ of nonnegative scalar curvature whose 
boundary is isometric to $(\Sigma,\gamma)$ with mean curvature $H$.  Considering the perturbed 
data $(\Sigma, \gamma, \lambda H)$ for a positive real parameter $\lambda$, we find that such a ``fill-in'' 
$(\Omega,g)$ must exist for $\lambda$ small and cannot exist for $\lambda$ large; moreover, we prove there 
exists an intermediate threshold value.

The main application is the construction of a new quasi-local mass, a concept of interest in general relativity.  
This mass has a nonnegativity property and is  bounded above by the Brown--York mass.  However,
our definition differs from many others in that it tends to vanish 
on static vacuum (as opposed to flat) regions.  We also recognize this mass as a special case of a type of 
twisted product of quasi-local mass functionals.  

\end{abstract}

\maketitle

\section{Introduction}
Riemannian 3-manifolds of nonnegative scalar curvature arise naturally in general relativity as totally geodesic
spacelike submanifolds of spacetimes obeying Einstein's equation and the dominant energy condition.  In this
setting, scalar curvature plays the role of energy density.  Black holes in this setting are 
manifested as connected minimal surfaces that minimize area to the outside.  
If $S$ is a disjoint union of such surfaces of total area $A$, the number $\sqrt{\frac{A}{16\pi}}$ is interpreted to encode the total mass of the collection of black holes, possibly accounting for potential energy between them \cite{bray_RPI}.

A fundamental question in general relativity is to quantify how much mass is contained in a compact region
$\Omega$ in a spacelike slice of a spacetime \cite{penrose}.  Constructing examples of such \emph{quasi-local mass} 
has led to a very active field of research (we mention here a small number of possible references: \cites{sza, wang_yau, 
imcf}).  For most definitions,
the quasi-local mass of $\Omega$ depends only boundary data of $\Omega$: namely the induced 2-metric
and induced mean curvature function.  We reference pioneering work of Bartnik \cites{bartnik_tsing, bartnik_mass}, 
whose name is given in the following definition.
 
All metrics and functions in this paper are assumed to be smooth, unless otherwise stated.
\begin{definition}
A triple $\Bcal=(\Sigma, \gamma, H)$, where $\Sigma$ is a topological 2-sphere, $\gamma$ is a Riemannian 
metric on $\Sigma$ of positive Gauss curvature, and $H$ is a positive function on $\Sigma$ is called 
\textbf{Bartnik data}.
\end{definition}
While not always necessary, it is often customary to restrict to positive Gauss curvature and positive 
functions $H$, as we do here.  A typical problem involving Bartnik data $(\Sigma, \gamma, H)$ is to construct a Riemannian
3-manifold $(M,g)$ satisfying some nice geometric properties such that the boundary $\partial M$ is isometric
to $(\Sigma, \gamma)$, and the mean curvature of $\partial M$ agrees with $H$.  For instance, one might
require $(M,g)$ to be asymptotically flat with nonnegative or zero scalar curvature (see \cite{bartnik_qs}, 
for instance).  Such a manifold is called an \emph{extension} of the Bartnik data.

We focus on the dual problem of constructing compact \emph{fill-ins} of the Bartnik data, realizing
$(\Sigma, \gamma, H)$ as the boundary of a compact 3-manifold.  This problem was considered by Bray in the construction of the Bartnik inner mass \cite{bray_RPI} (see section \ref{sec_inner_mass} below). 
\begin{definition}
A \textbf{fill-in} of Bartnik data $(\Sigma, \gamma, H)$ is a compact, connected Riemannian 3-manifold $(\Omega,g)$
with boundary such that there exists an isometric embedding $\iota:(\Sigma,\gamma) \to (\Omega, g)$ with the 
following properties:
\begin{enumerate}
 \item the image $\iota(\Sigma)$ is some connected component $S_0$ of $\partial \Omega$, and
 \item $H=H_{S_0} \circ \iota$ on $\Sigma$, where $H_{S_0}$ is the mean curvature of $S_0$ in $(\Omega, g)$.
\end{enumerate}
We adopt the sign convention that the mean curvature equals $-g(\vec H, \vec n)$, where $\vec H$ is the mean curvature
vector and $\vec n$ is the unit normal pointing out of $\Omega$ (e.g., the boundary of a ball in $\R^n$ has positive mean curvature).
\end{definition}
Without loss of generality, if $(\Omega, g)$ is a fill-in of $(\Sigma, \gamma, H)$,
we shall henceforth identify $\Sigma$ with $\iota(\Sigma)$ and $H$ with the mean curvature of $\iota(\Sigma)$.

We will primarily be concerned with fill-ins satisfying the following geometric constraints.
\begin{definition}
A fill-in $(\Omega,g)$ of $(\Sigma, \gamma, H)$ is \textbf{valid} if the metric $g$ has 
nonnegative scalar curvature and either
\begin{enumerate}
 \item $\partial \Omega=\Sigma$, or
 \item $\partial \Omega \sm \Sigma$ is a minimal (zero mean curvature) surface, possibly disconnected.
\end{enumerate}
\end{definition}
Figure \ref{fig_valid_fill_ins} provides a graphical depiction.
In physical terms, a valid fill-in is a compact region in a slice of a spacetime that has nonnegative energy 
density and possibly contains black holes.  Another characterization of the second class of valid fill-ins is 
a cobordism of nonnegative scalar curvature that joins the given Bartnik data to a minimal surface.  Note that
we require $\partial \Omega \sm \Sigma$ to be minimal, but not necessarily area-minimizing.  Figure
\ref{fig_valid_fill_ins} provides a graphical depiction.

\begin{figure}[ht]
\caption{Valid fill-ins of Bartnik data}
\begin{center}
\includegraphics[scale=0.7]{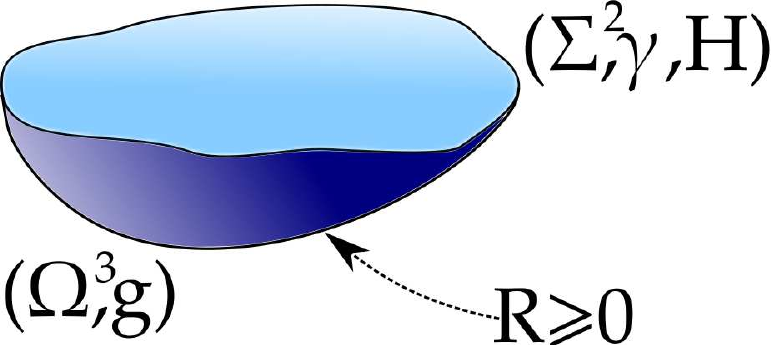}
\includegraphics[scale=0.3]{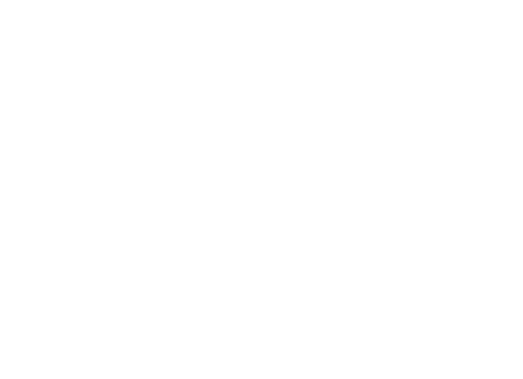}
\includegraphics[scale=0.7]{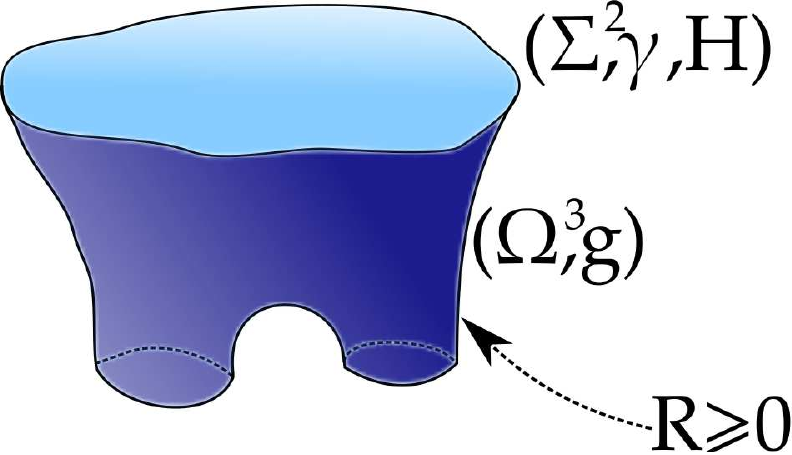}
\end{center}
\flushleft\footnotesize{On the left is a valid fill-in of $(\Sigma, \gamma, H)$ of the first type 
(i.e., $\partial \Omega = \Sigma$).  On the right is a valid fill-in of the second type ($\partial \Omega
\sm \Sigma$ is minimal). $R$ denotes the scalar curvature of $g$.}
\label{fig_valid_fill_ins}
\end{figure}

Interestingly, Bartnik data falls into one of three types.  Although trivial to prove, 
the following fact motivates much of the present paper.
\begin{obs}[Trichotomy of Bartnik data]  Bartnik data $(\Sigma, \gamma, H)$ belongs to exactly one
 of the following three classes:
\begin{enumerate}
 \item Negative type: $(\Sigma, \gamma, H)$ admits no valid fill-in.
 \item Zero type: $(\Sigma, \gamma, H)$ admits a valid fill-in, but every valid fill-in $(\Omega,g)$ has $\partial \Omega = \Sigma$.
 \item Positive type: $(\Sigma, \gamma, H)$ admits a valid fill-in $(\Omega, g)$ with nonempty minimal boundary $\partial \Omega \sm \Sigma$.
\end{enumerate}
\end{obs}

\begin{outline}
In section \ref{sec_fill_ins}, we give some geometric characterizations of valid fill-ins of Bartnik data
of zero and positive type, making connections with static vacuum metrics. We also recall in section
\ref{sec_inner_mass} the Bartnik inner mass, which explains the use of the words positive, zero, 
and negative in the trichotomy.

The essential idea of this paper, presented in section \ref{sec_interval}, is to study the behavior 
of Bartnik data $(\Sigma, \gamma, \lambda H)$, where the real parameter $\lambda > 0$ is allowed to vary.  
We show in Theorem \ref{thm_interval}, the main result, that the data passes through
all three classes of the trichotomy, with interesting behavior at some unique borderline value 
$\lambda=\lambda_0$.  In section \ref{sec_inner_mass_function}, we introduce a function that probes 
the geometry of valid fill-ins of $(\Sigma, \gamma, \lambda H)$.

The main application occurs in section \ref{sec_ql_mass}, where we use the number $\lambda_0$
to define a quasi-local mass for regions in 3-manifolds of nonnegative scalar curvature (Definition \ref{def_ql_mass}).  Several
properties are shown to hold, including nonnegativity.  
What distinguishes this definition from most others is its tendency to vanish on static vacuum,
as opposed to flat, data.  We give a brief physical argument for why such a property may be desirable in section \ref{sec_physical}.

Section \ref{sec_examples} consists of examples of Bartnik data of all three types,
and compares our definition with the Hawking mass and Brown--York mass.
In section \ref{sec_product} 
we introduce a general construction for ``twisting'' two quasi-local mass functionals
together, of which the above quasi-local mass is a special case.  The final section is a 
discussion of some potentially interesting open problems.
\end{outline}

\begin{ack}
The author is indebted to Hubert Bray for suggesting the main idea of varying the parameter $\lambda$.  He would like
to thank the referee for carefully reading the work and offering a number of thoughtful suggestions.
\end{ack}

\section{Fill-ins of nonnegative type and the inner mass}
\label{sec_fill_ins}

\subsection{Zero type data and static vacuum metrics}
First, we classify the geometry of valid fill-ins of Bartnik data of zero type.  Recall that a Riemannian
3-manifold $(\Omega,g)$ is \emph{static vacuum} if there exists a function $u\geq 0$ (called the \emph{static
potential}), with $u>0$ on the interior of $\Omega$, such that the Lorentzian metric
$$h = -u^2 dt^2 + g$$
on $\R \times \Int( \Omega)$ has zero Ricci curvature.  This condition is equivalent to the system of equations:
\begin{align}
\Delta u &=0 \label{eqn_static1}\\
u\Ric &= \Hess u \label{eqn_static2}
\end{align}
where $\Delta, \Ric$ and $\Hess$ are the Laplacian, Ricci curvature, and Hessian with respect to $g$.
Equation (\ref{eqn_static1}) together with the trace of (\ref{eqn_static2}) shows that 
static vacuum metrics have zero scalar curvature.  The following result is primarily a consequence of Corvino's work on local scalar 
curvature deformation \cite{corvino}.
\begin{prop}
\label{prop_static}
If $\Bcal$ is Bartnik data of zero type, then any valid fill-in is static vacuum.
\end{prop}
The idea of the proof is to use a valid fill-in that is not static vacuum to construct a valid fill-in 
that contains a black hole.  By a very rough analogy, one might think of this physically as taking some of the energy content
in a fill-in and squeezing it down into a black hole. The delicate issue is that we 
must preserve the boundary data in the process.

\begin{proof}
Let $(\Omega, g)$ be a valid fill-in of zero type data $(\Sigma, \gamma, H)$.  By definition, $\partial \Omega=\Sigma.$

We claim $g$ has identically zero scalar curvature.  If not, there exists $p \in \Int(\Omega)$ and $r>0$
such that on the closed metric ball $\ol B(p,r)$, the scalar curvature of $g$ is bounded below by some 
$\epsilon>0$.  On the set $\ol B(p,r/2) \sm \{p\}$, let $G$ be a Green's function for the Laplacian 
that blows up at $p$ and vanishes on $\partial \ol B(p,r/2)$ (see Theorem 4.17 of \cite{aubin}.) By the 
maximum principle, $G$ is positive, except on $\partial \ol B(p,r/2)$.  Extend $G$
by zero to the rest of $\Omega \sm \{p\}$, so that $G$ Lipschitz, smooth away from $\partial \ol B(p,r/2)$.  Perturb
$G$ to a smooth, nonnegative function $\tilde G$ on $\Omega \sm \{p\}$ that agrees with $G$ except possibly 
on the annular region $\ol B(p, 3r/4) \sm \ol B(p, r/4)$.  For a parameter $\delta > 0$ to be determined, define
the conformal metric
$$\tilde g = (1+ \delta \tilde G)^4 g$$
on $\Omega \sm \{p\}$.  By construction, $\tilde g = g$ outside $\ol B(p,r)$ and thus has nonnegative scalar
curvature outside this ball.  For points inside $\ol B(p,r)$, we apply the rule for the change in scalar
curvature under a conformal deformation (see appendix \ref{appendix_formulas}):
\begin{align*}
\tilde R &= (1+\delta \tilde G)^{-5}(-8\Delta(1 + \delta \tilde G) + (1+\delta \tilde G) R)\\
	&\geq (1+\delta \tilde G)^{-5} (-8\delta  \Delta \tilde G + \epsilon),
\end{align*}
Here, $\tilde R$ and $R$ are the scalar
curvatures of $\tilde g$ and $g$.  
Since $ \Delta \tilde G$ has compact support, we may choose $\delta>0$ sufficiently small so that the above is 
strictly positive.  In particular, $\tilde R \geq 0$ on $\Omega \sm \{p\}$.  

Now, suppose $s$ is the distance function with respect to $g$ from the point $p$.  For $s$ sufficiently
small, $G$ is of the form $\frac{c}{s}+O(1)$ for some constant $c>0$.  The normal derivative of $G$
to the sphere of radius $s$ about $p$ in the outward direction is $-\frac{c}{s^2} + O(s\inv)$ (see
Proposition 4.12 and Theorem 4.13 of \cite{aubin}).  The mean 
curvature of the sphere of radius $s$ with respect to $g$ is $\frac{2}{s}+O(1)$ (by Lemma 3.4 of \cite{fan_shi_tam}), and so the 
mean curvature of this sphere with respect to $\tilde g$ is (using appendix \ref{appendix_formulas}):
$$(1+ \delta \tilde G)^{-3} \left((2s\inv + O(1))(1+\delta c s\inv + O(1))-4\delta c s^{-2} + O(s\inv)\right).$$
The dominant term is $-2\delta c s^{-2}$, so that for some
$s>0$ sufficiently small, $\partial \ol B(p,s)$ has negative mean curvature (with respect to $\tilde g$) in direction
pointing away from $p$.   Let $\tilde \Omega$
be $\Omega \sm B(p,s)$, and restrict $\tilde g$ to $\tilde \Omega$.

The manifold $(\tilde \Omega, \tilde g)$ has boundary with two connected components, both of positive mean 
curvature in the \emph{outward} direction.  By Lemma \ref{lemma_pos_mc} below, $(\tilde \Omega, \tilde g)$ contains 
a subset that is a valid fill-in of $(\Sigma, \gamma, H)$ with a minimal boundary component. 
This contradicts the assumption that the Bartnik data is of zero type, and so we have proved the claim
that $g$ is scalar-flat.

Finally, if $(\Omega, g)$ is not static vacuum, then
Corvino proves the existence of a metric $\ol g$ on $\Omega$ with nonnegative, scalar curvature, positive
at some interior point $p$, such that $g - \ol g$ is supported away from $\partial \Omega$ \cite{corvino}.  
In particular,
$(\Omega, \ol g)$ is a valid fill-in for the type-zero data $(\Sigma,\gamma, H)$, and the above argument leads to a contradiction.
\end{proof}

To complete the proof of the previous proposition, we have the following lemma:
\begin{lemma}
Suppose $\Bcal=(\Sigma, \gamma,H)$ admits a fill-in $(\Omega,g)$ with nonnegative scalar curvature, such 
that $\partial \Omega \sm \Sigma$ has positive mean curvature in the outward direction.  Then a subset
$\Omega'$ of $\Omega$ is a valid fill-in of $\Bcal$ with metric $g|_{\Omega'}$.  Moreover,
$\Omega'$ has at least one minimal boundary component.
\label{lemma_pos_mc}
\end{lemma}
\begin{proof}
By assumption $\Sigma$ has positive mean curvature $H$ and $\partial \Omega \sm \Sigma$ has positive mean
curvature.  By Theorem \ref{gmt_lemma} in appendix \ref{appendix_gmt}, there exists a smooth, embedded minimal 
surface $S$ homologous to $\Sigma$.  The closure of the region bounded between $\Sigma$ and $S$ is the desired 
valid fill-in.
\end{proof}

\subsection{Data of positive type}
\begin{prop} 
\label{prop_positive}
Given Bartnik data $\Bcal$, the following are equivalent:
\begin{enumerate}
\item $\Bcal$ is of positive type.
\item $\Bcal$ admits a valid fill-in that has positive scalar curvature at some point.
\item $\Bcal$ admits a valid fill-in that has positive scalar curvature everywhere.
\end{enumerate}
\end{prop}
The idea of proving the proposition is to create positive energy density at some interior points
at the expense of decreasing the size of the minimal surface.  As in the previous section, the delicate issue is preserving the boundary data in the process.

\begin{proof}
If $\Bcal$ admits a valid fill-in with positive scalar curvature at a point, then $\Bcal$ is of nonnegative
type and Proposition \ref{prop_static} rules out the case of zero type (since static vacuum metrics have zero scalar curvature).  This shows $(2)$ implies $(1)$;
$(3)$ trivially implies $(2)$.

Last, we show $(1)$ implies $(3)$.  
Suppose $\Bcal=(\Sigma, \gamma, H)$ has positive type, so there exists some valid fill-in $(\Omega,g)$ of $\Bcal$
with boundary $\Sigma \dot \cup S$, where $S$ is a nonempty minimal surface.  If $(\Omega,g)$ is not static 
vacuum, we may complete the proof by again using the work of Corvino to perturb $(\Omega,g)$ to a valid fill-in with positive scalar 
curvature at a point \cite{corvino}.  Thus, assume $(\Omega,g)$
is static vacuum, and so in particular it is scalar-flat.

Replace $(\Omega,g)$ with its double across the minimal surface $S$.  Now, $(\Omega,g)$
has two boundary components $\Sigma$ and $\Sigma'$ (its reflected copy), and contains
a minimal surface $S$ that is fixed by the $\mathbb{Z}_2$ reflection symmetry.  Moreover, $g$
is Lipschitz continuous across $S$ and smooth elsewhere\footnote{This doubling trick across a minimal
surface was used by Bunting and Masood-ul-Alam to classify static vacuum metrics
with compact minimal boundary that are asymptotically flat \cite{bma}.  Because of the asymptotic condition, 
their theorem does not apply to the present case.  We also mention the fact that because of minimality
and the static vacuum condition, $S$ is totally  geodesic, which implies that $\tilde g$ is $C^{1,1}$ 
across $S$ \cites{bma, corvino}.  However, we do not need this fact.}.  For simplicity of exposition, we separately treat the cases in which $g$ is smooth and non-smooth across $S$.

\emph{Smooth case:  }
For $\epsilon \in(0,1)$, let $\varphi$ be the function on $\Omega$ solving the following Dirichlet problem:
\begin{equation*}
 \begin{cases}
  \Delta \varphi = 0 & \text{ on } \Omega,\\
  \varphi = 1 & \text { on } \Sigma,\\
  \varphi = 1-\epsilon & \text { on } \Sigma'.
 \end{cases}
\end{equation*}
Consider the conformal metric $\tilde g=\varphi^4 g$, which is smooth with zero scalar curvature.  Moreover, the mean curvature $\tilde H$ of $\Sigma$ with respect to $\tilde g$
strictly exceeds $H$ (for all choices of $\epsilon$), since $\varphi$ has positive outward normal derivative on $\Sigma$
(see appendix \ref{appendix_formulas}).  The mean curvature of $\Sigma'$ remains positive 
for $\epsilon >0$ sufficiently small.  Fix such an $\epsilon$.  This construction is demonstrated in figure
\ref{fig_reflection}.

\begin{figure}[ht]
\caption{Construction in proof of Proposition \ref{prop_positive}}
\begin{center}
\includegraphics[scale=0.7]{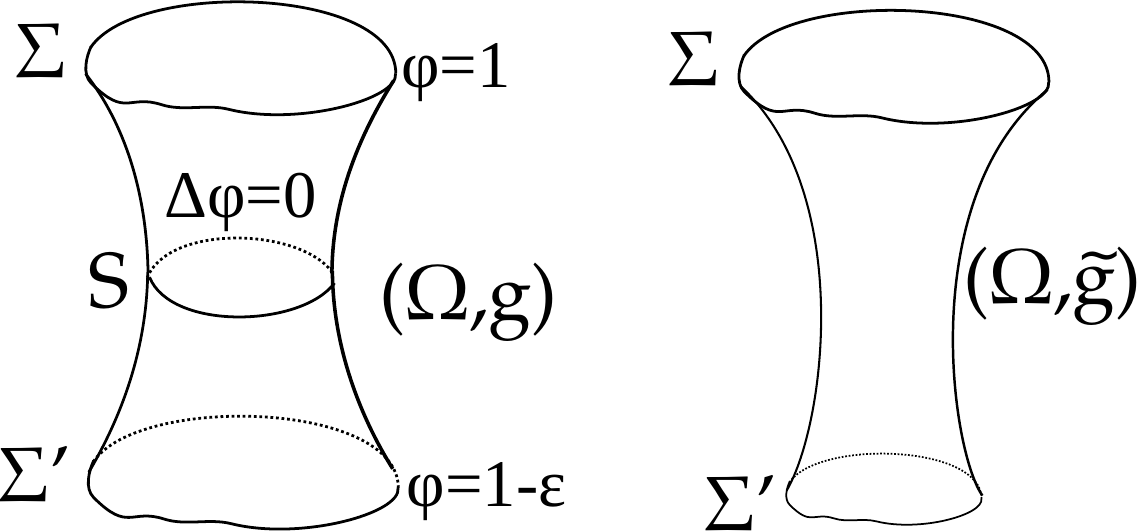}
\end{center}
\flushleft\footnotesize{On the left is the double of $(\Omega,g)$, which we also refer to as $(\Omega,g)$,
abusing notation. The function $\varphi$ is harmonic, with the given prescribed Dirichlet boundary values.
On the right is $\Omega$ equipped with the metric $\tilde g$, obtained from $g$ by applying the conformal
factor $\varphi^4$.}
\label{fig_reflection}
\end{figure}

Fix any smooth function $\rho>0$ on $\Omega$.  For all $\delta\geq 0$ small, let $u_\delta$ 
be the unique solution to the elliptic problem
\begin{equation*}
 \begin{cases}
  \tilde L u_\delta = \delta \rho & \text{ in } \tilde \Omega,\\
  u_\delta = 1 & \text{ on } \Sigma,\\
  \partial_\nu (u_\delta) = 0 & \text{ on } \Sigma',
 \end{cases}
\end{equation*}
where $\tilde L = -8\tilde\Delta$ is the conformal Laplacian of $\tilde g$.  Clearly $u_0 \equiv 1$,
and $u_\delta$ converges in $C^2$ to 1 as $\delta \to 0^+$.   For $\delta>0$ small enough to ensure $u_\delta > 0$,
the conformal metric $u_\delta^4 \tilde g$ has:
\begin{itemize}
\item positive scalar curvature (equal to $\delta\rho u_\delta^{-5}$),
\item induced metric on $\Sigma$ equal to $\gamma$ (by the boundary condition $u_\delta|_\Sigma=1$),
\item positive mean curvature on $\Sigma'$ (by the boundary condition $\partial_{\nu} (u_\delta)|_{\Sigma'}=0$), and
\item mean curvature on $\Sigma$ converging uniformly to $\tilde H$  as $\delta \to 0^+$.  
\end{itemize}
Fix a particular value of $\delta$ such that the mean curvature of $\Sigma'$ is positive
 and the mean curvature $\tilde H_\delta$ of $\Sigma$ is pointwise greater than $H$ 
(which is possible, since $\tilde H > H$).  By Lemma \ref{lemma_pos_mc}, there is a valid fill-in of $(\Sigma,\gamma,\tilde H_\delta)$
that contains a minimal surface.  By Lemma \ref{lemma_positive} in appendix \ref{appendix_bmn}, this valid fill-in can be perturbed to a valid fill-in of $(\Sigma, \gamma,H)$ so that the latter fill-in still has positive scalar curvature.

\emph{Lipschitz case:  } In general we must carry out an extra step to deal with the lack of smoothness across $S$.  
Define $\varphi$ analogously by first solving $\Delta \varphi_1=0$
with boundary conditions of $1$ on $\Sigma$ and $1-\frac{\epsilon}{2}$ on $S$, then
defining $\varphi_2 = 2 - \epsilon - \varphi_1$ in the reflected copy.  
The function $\varphi$ obtained by gluing $\varphi_1$ and $\varphi_2$
is $C^{1,1}$ on $\Omega$, and smooth and harmonic away from $S$.  Again, let $\tilde g=\varphi^4 g$, which has zero scalar curvature (away from $S$), is Lipschitz across $S$, and induces the same mean curvature on both sides of $S$.  Fix $\epsilon>0$ so that $\tilde H>H$ and the $\tilde g$-mean curvature of $\Sigma'$ is positive.

By the work of Miao \cite{miao}, the fact that both sides of $S$ have the same mean curvature implies the existence of a family of $C^2$ 
metrics $\{\tilde g_\delta\}_{0 < \delta < \delta_0}$ such that
\begin{enumerate}
 \item $\tilde g_\delta$ converges to $\tilde g$ in $C^0$ as $\delta \to 0^+$,
 \item $\tilde g_\delta$ agrees with $\tilde g$ outside a $\delta$-neighborhood of $S$, and
 \item the scalar curvature $\tilde R_\delta$ of $\tilde g_\delta$ is bounded below by a constant 
 independent of $\delta$.
\end{enumerate}
In particular, the $L^p$ norm of $\tilde R_\delta$ (taken with respect to $\tilde g$ or $\tilde g_\delta$)
for any $1 \leq p < \infty$ converges to zero as $\delta \to 0$.  We mimic arguments 
of Schoen and Yau \cite{schoen_yau} to prove:
\begin{lemma}
For each $\delta>0$ sufficiently small, the conformal Laplacian $\tilde L_\delta = 
-8\tilde \Delta_\delta  + \tilde R_\delta$ of $\tilde g_\delta$ has trivial kernel on the space of functions 
$v$ with boundary conditions of $v=0$ on $\Sigma$ and $\partial_\nu v = 0$ on $\Sigma'$.
\end{lemma}
\begin{proof}
Let $v$ belong to the kernel of $\tilde L_\delta$ with the above boundary conditions.  Multiplying $\tilde L_\delta v$
by $v$ and integrating by parts gives
$$0 = \int_\Omega \left(8|\nabla v|^2_{\tilde g_\delta} + \tilde R_{\delta} v^2\right) d\tilde V_{\delta}.$$
Let $\tilde R_\delta^- = -\min(\tilde R_\delta, 0)$, so that
\begin{align*}
\int_\Omega 8|\nabla v|^2_{\tilde g_\delta} d\tilde V_{\delta} &\leq \int_\Omega \tilde R_\delta^- v^2 d\tilde V_{\delta}\\
 &\leq \left(\int_\Omega  (\tilde R_\delta^-)^{3/2}  d\tilde V_{\delta}\right)^{2/3} \left(\int_\Omega  v^6  d\tilde V_{\delta}\right)^{1/3}\\
 &\leq c\left(\int_\Omega  (\tilde R_\delta^-)^{3/2}  d\tilde V_{\delta}\right)^{2/3} \left(\int_\Omega  |\nabla v|_{\tilde g_\delta}^2  d\tilde V_{\delta}\right),
\end{align*}
having used the H\"{o}lder and Sobolev inequalities (where $c>0$ is a constant).  Thus, for $\delta$ sufficiently
small, a nonzero $v$ may not exist, since the $L^{3/2}$ norm of $\tilde R_\delta^-$ converges to zero.
\end{proof}

Fix a smooth function $\rho>0$ on $\Omega$.  By the lemma and standard elliptic theory, for $\delta>0$
small there exists unique solution $u_\delta$ to the problem:
\begin{equation*}
 \begin{cases}
  \tilde L_\delta u_\delta = \delta \rho & \text{ in } \tilde \Omega,\\
  u_\delta = 1 & \text{ on } \Sigma,\\
  \partial_\nu (u_\delta) = 0 & \text{ on } \Sigma'.
 \end{cases}
\end{equation*}
A key fact is that $u_\delta$ converges to 1 in $C^0$  as $\delta \to 0^+$, and this convergence
is $C^2$ away from $S$ (see the proof of Proposition 4.1 of Miao \cite{miao}).

At this point, the proof follows nearly the same steps as in the smooth case, where we work with the metric $u_\delta^4 \tilde g_\delta$ (which has positive scalar curvature and induces the metric $\gamma$ on $\Sigma$).  We pick $\delta>0$ sufficiently small so that $\tilde H_\delta>H$
and $\Sigma'$ has positive mean curvature with respect to $u_\delta^4 \tilde g_\delta$.  Now, if necessary, perturb the $C^2$ metric $u_\delta^4 \tilde g_\delta$ on a neighborhood of $S$ to a $C^\infty$ metric,
preserving the above properties.  The proof now goes as in the smooth case, making use of Lemmas \ref{lemma_pos_mc} and \ref{lemma_positive}.
\end{proof}

We remark that our assumption of positive Gauss curvature of $(\Sigma, \gamma)$ is not necessary in Propositions \ref{prop_static}
and \ref{prop_positive}.

\subsection{Bartnik inner mass}
\label{sec_inner_mass}
One source of inspiration for the problem of considering valid fill-ins with minimal boundary is Bray's
definition of the Bartnik inner mass \cite{bray_RPI}, an example of a quasi-local mass (see section \ref{sec_ql_mass} for
more on quasi local mass).  The Bartnik inner mass aims to measure the size of the largest 
black hole  that could be placed inside a valid fill-in of given Bartnik data.
\begin{definition}
The \textbf{Bartnik inner mass} of Bartnik data $\Bcal$
is the real number
$$m_{inner}(\Bcal) = \sup_{(\Omega,g)} \left\{\sqrt{\frac{A}{16\pi}}\right\}$$
where the supremum is taken over the class of all valid fill-ins $(\Omega,g)$ of $\Bcal$, and $A$ is the minimum
area in the homology class of $\Sigma$ in $(\Omega,g)$.
\end{definition}
This definition, though formulated differently, is equivalent to Bray's.  The purpose of using the minimum
area in the homology class of $\Sigma$ is to ignore any large minimal surfaces ``hidden behind'' a smaller
minimal surface.

We observe that the sign of $m_{inner}(\Bcal)$ corresponds directly to the type of the Bartnik data $\Bcal$.  To see this,
first note that for fill-ins with a minimal boundary, the minimum area of $A$
in the homology class of $\Sigma$ in $(\Omega,g)$ is always attained by a smooth minimal surface, 
and so $A$ is positive (see Theorem \ref{gmt_lemma}).  For fill-ins without boundary, $\Sigma$ is 
homologically trivial, and so $A=0$.  Thus, $m_{inner}(\Bcal)$ is positive if $\Bcal$ is of positive type;
zero if $\Bcal$ is of zero type; and $-\infty$ if $\Bcal$ is of negative type.

\section{The interval of positivity}
\label{sec_interval}
The following idea was suggested by Bray: as a function of a parameter $\lambda > 0$, consider
the Bartnik data $(\Sigma, \gamma, \lambda H)$.  The main purpose of this section is to state and prove 
Theorem \ref{thm_interval}, which partially answers the question of how the type of the data
depends on $\lambda$.

One key ingredient is the following well-known theorem of Shi and Tam\footnote{We remark that the Shi--Tam theorem
was originally stated for the case in which every component of $\partial\Omega$ has positive Gauss and mean curvatures. 
However, one can allow additional minimal surface components (as we have done here) by observing
the positive mass theorem is true for manifolds with compact minimal boundary.  Alternatively, one could
employ a reflection argument to eliminate any minimal surface boundary components.}.
\begin{thm}[Shi--Tam, 2002 \cite{shi_tam}]
\label{thm_shi_tam}
If Bartnik data $(\Sigma, \gamma, H)$ has a valid fill-in $(\Omega, g)$, then
\begin{equation}
\int_\Sigma (H_0 - H) dA_\gamma \geq 0,
\label{eqn_shi_tam}
\end{equation}
where $H_0$ is the mean curvature of an isometric embedding of $(\Sigma, \gamma)$ into Euclidean space $\R^3$, and
$dA_\gamma$ is the area form on $\Sigma$ with respect to the metric $\gamma$.  Moreover, equality holds if and
only if $(\Omega,g)$ is isometric to a subdomain of $\R^3$.
\end{thm}
Recall that we assume $\gamma$ to have positive Gauss curvature, which is necessary for the theorem:
$H_0$ is well-defined, since an isometric embedding of a positive Gauss curvature surface into $\R^3$
exists and is unique up to rigid motions (see the references in \cite{shi_tam}).  

In our case, inequality (\ref{eqn_shi_tam}), which depends only on the Bartnik data, must be satisfied for data 
that admits a valid  fill-in.  In particular, by increasing $H$ (while keeping $\gamma$, and therefore $H_0$, fixed), it is clear 
that some Bartnik data do not possess fill-ins (i.e., are of negative type).  Hence, the Shi--Tam theorem
gives an obstruction to Bartnik data being of nonnegative type.

The following main theorem demonstrates that there exists a unique interval of values
of $\lambda$ for which this data $(\Sigma,\gamma,\lambda H)$ is of positive type.

\begin{thm}
Fix Bartnik data $(\Sigma, \gamma, H)$.  There exists a unique number $\lambda_0 > 0$ such that
$(\Sigma, \gamma, \lambda H)$ is of positive type if and only if $\lambda \in (0, \lambda_0)$.  Moreover,
$(\Sigma, \gamma, \lambda H)$ is of negative type if $\lambda > \lambda_0$.
\label{thm_interval}
\end{thm}
As a consequence, $(\Sigma, \gamma, \lambda H)$ is zero type for at most one value of $\lambda$, namely
$\lambda_0$.
\begin{proof}
Define
\begin{align*}
I_+ &= \{\lambda \in \R^+ : (\Sigma, \gamma, \lambda H) \text{ is of positive type}\},\\
I_0 &= \{\lambda \in \R^+ : (\Sigma, \gamma, \lambda H) \text{ is of zero type}\},\\
I_{\geq 0} &= I_+ \cup I_0.
\end{align*}

\emph{Step 1:  } We first show $I_+$ is nonempty.  Consider the space $\Omega = \Sigma \times [-1,0]$ with
product metric $g$, and identify $\Sigma$ with $\Sigma \times \{0\}$.  Let $S$ be the other boundary component
of $\Omega$, namely 
$\Sigma \times \{-1\}$.  Observe that 1) $\Omega$ has positive scalar curvature since $\Sigma$ has
positive Gauss curvature, and 2) all leaves $\Sigma \times \{t\}$ are minimal surfaces.

Choose a smooth function $v$ on $\Omega$ satisfying the following properties: $v \leq 0$, $v$
vanishes on $\Sigma$ and in a neighborhood of $S$, and $\partial_{\nu}v = \frac{1}{4} H$ on $\Sigma$.
For $\epsilon > 0$, let $u_\epsilon = 1 + \epsilon v$.  In particular,
$u_\epsilon$ is positive for $\epsilon>0$ sufficiently small.  Consider the conformal metric $g_\epsilon
=u_\epsilon^4 g$.  Note that $g_\epsilon$ induces the metric $\gamma$ on $\Sigma$, and assigns
the following value to the mean curvature of $\Sigma$:
$$H_\epsilon = 4\partial_\nu(u_\epsilon)=4\epsilon\partial_\nu v=\epsilon H,$$
by our choice of $v$.  Moreover, the scalar curvature of $g_\epsilon$ is
\begin{align*}
R_{g_\epsilon} &= u_{\epsilon}^{-5} \left(-8\Delta_g u_\epsilon + R_g u_\epsilon\right)\\
	&= u_{\epsilon}^{-5} \left(-8\epsilon \Delta_g v + R_g u_\epsilon\right),
\end{align*}
which is positive for $\epsilon$ sufficiently small, 
since $R_g > 0$ and $u_\epsilon$ is uniformly
bounded below as $\epsilon \to 0^+$.  Fix such an $\epsilon$. 
We can see $(\Omega, g_\epsilon)$ is a valid fill-in of $(\Sigma, \gamma, \epsilon H)$,
since this fill-in has positive scalar curvature, induces the correct boundary geometry on $\Sigma$,
and $S$ is minimal (since $g_\epsilon=g$ near $S$). In particular, $\epsilon$ belongs
to $I_+$, so $I_+ \neq \emptyset$.

\emph{Step 2:  } The next step is to show that $I_{\geq 0}$ is connected, and $I_0$ contains at most one point.
To accomplish this, we show that for every number in $I_{\geq 0}$, every smaller
positive number belongs to $I_+$.  It suffices to show that if $(\Sigma, \gamma, H)$ is of nonnegative type,
then $(\Sigma, \gamma, \lambda H)$ is of positive type for all $\lambda \in (0,1)$.  This fact follows from
the next lemma, by  Proposition \ref{prop_positive}.
\begin{lemma}
\label{lemma_stretch}
Let $(\Omega,g)$ be a fill-in of arbitrary Bartnik data $(\Sigma,\gamma,H)$.  Fix $\lambda \in (0,1)$ and a neighborhood $U$ of $\Sigma$ in $\Omega$.  There exists
a metric $\tilde g$ on $\Omega$ such that:
\begin{enumerate}[(a)]
\item $\tilde g$ is a fill-in of $(\Sigma, \gamma, \lambda H)$,
\item $\tilde g \geq g$, with equality outside $U$, and
\item $R_{\tilde g} \geq \min(0, R_g)$ pointwise, with strict inequality on a neighborhood of $\Sigma$, where $R$ and $R_{\tilde g}$
are the scalar curvature of $g$ and $\tilde g$.
\end{enumerate}
In particular, if $(\Omega, g)$ is a valid fill-in, so is $(\Omega, \tilde g)$.
\end{lemma}

\begin{proof} Working in a neighborhood of $\Sigma$ in $\Omega$ diffeomorphic to $(-t_0,0] \times \Sigma$ and contained in $U$, we may assume $g$ takes the form
$$g = dt^2 + G_t,$$
where $t$ is the negative of $g$-distance to $\Sigma$, and $G_t$ is a Riemannian metric on the surface $\Sigma_t = \Sigma \times \{t\}$.  Shrinking $t_0$ if necessary, we may assume that every $(\Sigma_t,G_t)$ has positive Gauss curvature $K_t$
and positive mean curvature $H_t$ (in the outward direction $\partial_t$).  Let $\rho:(-t_0,0] \to \R$ be a smooth
function satisfying
\begin{enumerate}
 \item $\rho \equiv 1$ in a neighborhood of $-t_0$, 
 \item $\rho(0) =\lambda\inv > 1$,
 \item $\rho'(t) \geq 0$, and
 \item $\rho'(0) > 0$.
\end{enumerate}
Define a new metric $\tilde g$ on $\Omega$ by setting
\begin{equation}
\label{eqn_warped_metric}
\tilde g = \rho(t)^2 dt^2 + G_t
\end{equation}
on the neighborhood of $\Sigma$, and extending smoothly by $g$ to the rest of $\Omega$; claim (b) is satisfied.  A straightforward
calculation shows that $\Sigma$ has mean curvature $\lambda H$ in the metric $\tilde g$; moreover $\tilde g$
induces the metric $\gamma$ on $\Sigma$, so claim (a) holds.  
Last, we must study the scalar curvature of $\tilde g$ on the neighborhood $(-t_0,0] \times \Sigma$.
The following well-known formula, obtained from computing
the variation of mean curvature under a unit normal flow, 
 gives the
scalar curvature of $g$ as:
\begin{equation}
\label{eq_second_variation}
R_g = -2 \frac{\partial H_t}{\partial t} + 2K_t - H_t^2 - \|h_t\|^2,
\end{equation}
where $h_t$ is the second fundamental form of $\Sigma_t$ in $(\Omega, g)$, and its norm $\|\cdot\|^2$ is
taken with respect to $G_t$.  Applying this formula to the metric $\tilde g$ yields
\begin{equation}
\label{eqn_R_tilde_g}
 R_{\tilde g} = \frac{1}{\rho(t)^2} R_g + 2K_t(1 - \rho(t)^{-2}) + 2\frac{\rho'(t)}{\rho(t)^3}H_t.
\end{equation}
Now, $K_t > 0, \rho(t) \geq 1, \rho'(t) \geq 0$
and $H_t > 0$, so we see $R_{\tilde g}(x) \geq 0$ if $R_g(x) \geq 0$ and $R_{\tilde g}(x) \geq R_g(x)$
if $R_g(x) < 0$; both are strict inequalities near $t=0$, proving claim (c).
\end{proof}
We conclude that $I_{\geq 0}$ is a convex subset of $\R^+$, containing all arbitrarily small positive numbers.
Moreover, $I_0$ contains at most a single point.

\emph{Step 3:  } We prove that $I_{\geq 0}$ is bounded above. This follows immediately from the work of Shi and Tam.
More precisely, if $\lambda \in I_{\geq 0}$, then $$\lambda \leq \frac{\int_\Sigma H_0 dA_\gamma}{\int_\Sigma H dA_\gamma}.$$
Together with step 2, we see $I_{\geq 0}$ and $I_+$ are intervals of the form $(0, \lambda_0]$ or $(0,\lambda_0)$.

\emph{Step 4:  } Here we prove that $\lambda_0$ does not belong to $I_+$.  If $\lambda_0 \in I_+$, then by 
Proposition \ref{prop_positive}, there exists
a valid fill-in $(\Omega,g)$ of $(\Sigma, \gamma, \lambda_0 H)$ with positive scalar curvature at some point
and boundary $\Sigma \dot \cup S_0$, with $S_0$ minimal and nonempty. 
Solve the mixed Dirichlet--Neumann problem:
\begin{equation}
\begin{cases}
 \Delta u = \frac{1}{8} R_g u & \text{ in } \Omega,\\
 u = 1 & \text{ on } \Sigma,\\
 \partial_\nu(u) = 0 & \text{ on } S_0.
\end{cases}
\end{equation}
Here, $\nu$ is the unit normal, always chosen to point out of $\Omega$.  Note that a solution
exists because $R_g \geq 0$. 
By the maximum principle, $u>0$ in $\Omega$ and $\partial_{\nu}(u) > 0$ on $\Sigma$.  Let $g'=u^4 g$.  Note that $g'$ has zero scalar curvature, induces the metric $\gamma$ on $\Sigma$
and assigns zero mean curvature to $S_0$.  In particular, if we let $H'$ be the mean curvature
of $\Sigma$ with respect to $g'$, then $(\Sigma, \gamma, H')$ has a valid fill-in with minimal boundary, namely $(\Omega, g')$,
and is therefore of positive type.  Observe that $H'> \lambda_0 H$.  Choose $\beta > 1$ so that 
$H' > \beta\lambda_0 H$.  By Lemma \ref{lemma_positive} in appendix \ref{appendix_bmn}, we see that $(\Sigma, \gamma, \beta \lambda_0 H)$ 
is of positive type. Therefore $\beta \lambda_0 \in I_+$, which contradicts $\lambda_0 = \sup I_+$.
We conclude $I_+ = (0, \lambda_0)$, and either $I_{\geq 0} = (0,\lambda_0)$ or $(0,\lambda_0]$.
It follows that if $\lambda>\lambda_0$, then $(\Sigma, \gamma, \lambda H)$ must be of negative type.
\end{proof}

To emphasize the picture, the data $(\Sigma, \gamma, \lambda H)$ is of positive type for $\lambda$ small.
As we increase $\lambda$, this behavior persists until $\lambda = \lambda_0$.  At this point, the data
is zero or negative, and for $\lambda > \lambda_0$, the data is negative.  See section \ref{sec_open_problems}
for further discussion of the behavior near $\lambda=\lambda_0$.

\subsection{Inner mass function}
\label{sec_inner_mass_function}
In the remainder of this section we will study 
the function
\begin{equation}
m(\lambda) = m_{inner}(\Sigma, \gamma, \lambda H)
 \label{eqn_m_lambda}
\end{equation}
defined for $\lambda \in (0, \lambda_0)$.  Intuitively, one would
expect the following behavior of the function $m(\lambda)$. 
For $\lambda$ small, the mean curvature $\lambda H$
is close to zero, so one might anticipate the existence of a valid fill-in
with minimal boundary of approximately the same area as $\Sigma$.

As $\lambda$ increases, one would expect the class of valid fill-ins to shrink; one
reason is that the Shi--Tam inequality is more difficult to satisfy.  Consequently, the Bartnik inner mass 
ought to decrease as well.  The following statement supports this intuition.
\begin{prop}
Given Bartnik data $(\Sigma, \gamma, H)$, the function $m:(0, \lambda_0) \to \R^+$ is continuous and decreasing,
with the following limiting behavior:
$$\lim_{\lambda \to 0^+} m(\lambda) = \sqrt{\frac{|\Sigma|_\gamma}{16\pi}}.$$
Here, $|\Sigma|_\gamma$ is the area of $\Sigma$ with respect to $\gamma$.
\label{prop_m_lambda}
\end{prop}

\begin{proof}
\emph{Monotonicity:  }
Given $0 < \lambda_1 < \lambda_2 < \lambda_0$, we showed in Lemma \ref{lemma_stretch} 
that any valid fill-in of $(\Sigma, \gamma, \lambda_2 H)$ gives rise to a valid fill-in of 
$(\Sigma, \gamma, \lambda_1 H)$ with a metric that is pointwise at least as large (see (\ref{eqn_warped_metric})).  
From the definition of the Bartnik inner mass, this shows that
$$m(\lambda_1) \geq m(\lambda_2).$$

\emph{Continuity:  } Suppose $0 < \lambda_1 < \lambda_0$, and let $\epsilon > 0$.  From
the definition of the Bartnik inner mass, there exists a valid fill-in $(\Omega,g)$ of
$(\Sigma, \gamma, \lambda_1 H)$ whose minimum area $A$ in the homology class of $\Sigma$
satisfies
$$m(\lambda_1) - \sqrt{\frac{A}{16\pi}} < \frac{\epsilon}{3}.$$
From Proposition \ref{prop_positive}, there exists a valid fill-in $(\tilde \Omega,\tilde g)$ of $(\Sigma, \gamma, \lambda_1 H)$ that has strictly positive scalar curvature, and whose minimum area $\tilde A$ in the homology class of $\Sigma$ is close to $A$:
$$\sqrt{\frac{A}{16\pi}} - \sqrt{\frac{\tilde A}{16\pi}} < \frac{\epsilon}{3}.$$
Now, for $\lambda_0>\lambda>\lambda_1$, $(\tilde \Omega, \tilde g)$
can be perturbed to a fill-in $(\tilde \Omega, \tilde g_\lambda)$ of $(\Sigma, \gamma, \lambda H)$
using a metric of the form $(\ref{eqn_warped_metric})$.  The scalar curvature of $\tilde g_\lambda$ has  
potentially decreased relative to that of $\tilde g$, but remains positive for $\lambda > \lambda_1$ sufficiently close to $\lambda_1$.  Since $\tilde g_\lambda \to \tilde g$ in $C^0$, we may assume $\lambda - \lambda_1$ is small enough so that
$$\sqrt{\frac{\tilde A}{16\pi}} - \sqrt{\frac{\tilde A_\lambda}{16\pi}} < \frac{\epsilon}{3},$$
where $\tilde A_\lambda$ is the minimum $\tilde g_\lambda$-area in the homology class of 
$\Sigma$.  Adding the last three inequalities and using the definition of the Bartnik inner mass gives
\begin{align*}
m(\lambda_1) &< \epsilon + \sqrt{\frac{\tilde A_\lambda}{16\pi}}\\ 
&\leq \epsilon + m(\lambda)
\end{align*}
for $\lambda - \lambda_1$ sufficiently small.  Together with the fact that $m(\cdot)$ is decreasing,
we have shown $m(\cdot)$ is continuous at $\lambda_1$.

\emph{Lower limit behavior:  }
To study the behavior of $m(\epsilon)$ for $\epsilon$ small,
recall that in Step 1 of the proof of Theorem \ref{thm_interval} we constructed a valid fill-in of $(\Sigma, \gamma,
\epsilon H)$ by a metric $g_\epsilon$ uniformly close (controlled by $\epsilon$) to a cylindrical product metric $g$
over $(\Sigma, \gamma)$.  As $\epsilon \to 0^+$, the minimum $g_\epsilon$ area in the homology class of $\Sigma$
converges to the minimum $g$-area in the same homology class, which is $|\Sigma|_\gamma$.  On the
other hand, the Bartnik inner mass of $(\Sigma, \gamma, H')$ (for any $H'$) never exceeds 
$\sqrt{\frac{|\Sigma|_\gamma}{16\pi}}$ by definition.
This proves
$$\lim_{\lambda \to 0+} m(\lambda) = \sqrt{\frac{|\Sigma|_\gamma}{16\pi}}.$$
\end{proof}
In section \ref{sec_open_problems} we conjecture that $m(\lambda)$ limits to zero as $\lambda \to \lambda_0^-$, 
behavior supported by the explicit computation of $m(\lambda)$ in a spherically-symmetric case in section 
\ref{sec_example_m_lambda}.

\section{Quasi-local mass}
\label{sec_ql_mass}
Recall from the introduction the problem of assigning a ``quasi-local mass'' to a bounded region $\Omega$
in a totally geodesic spacelike slice $(M,g)$ of a spacetime.  By most definitions, the quasi-local
mass of $\Omega$ depends only on the Bartnik data $(\Sigma, \gamma, H)$ of the boundary, and we adopt this
perspective here.  That is, we define a \emph{quasi-local mass functional} to be a map from 
(a subspace of) the set of Bartnik data to the real numbers. We refer the reader to \cite{sza} for a recent comprehensive survey of 
quasi-local mass.

We begin by recalling some well-known examples of quasi-local mass.  First, 
the \emph{Hawking mass} of $(\Sigma, \gamma, H)$ is defined to be
$$m_{H}(\Sigma, \gamma, H) = \sqrt{\frac{|\Sigma|_\gamma}{16\pi}}\left(1 -\frac{1}{16\pi} \int_\Sigma H^2 dA_\gamma\right).$$
There is no correlation between the sign of the Bartnik data and the sign of the Hawking mass.  That is,
the Hawking mass can be negative for positive Bartnik data, and vice versa (see section \ref{sec_examples}).

Next, the \emph{Brown--York mass} is defined for Bartnik data $(\Sigma, \gamma, H)$ (assuming as we do that $K_\gamma>0$ and $H>0$) by
\begin{equation*}
m_{BY}(\Sigma, \gamma, H) = \frac{1}{8\pi} \int_\Sigma (H_0 - H) dA_\gamma,
\end{equation*}
where $H_0$ is the mean curvature of an isometric embedding of $(\Sigma, \gamma)$ into $\R^3$.  
Theorem \ref{thm_shi_tam} of Shi--Tam establishes that the Brown--York mass is nonnegative for 
Bartnik data of nonnegative type.  However, there exist Bartnik data of both negative and zero type for which the 
Brown--York mass is strictly positive (see section \ref{sec_examples}).

A third example is the Bartnik inner mass, defined in section \ref{sec_inner_mass}.

A key observation is that Theorem \ref{thm_interval} canonically associates to any Bartnik
data (with $H>0$ and $K_\gamma >0$) a positive number $\lambda_0$, which we call the \emph{critical parameter}.  
In this section we use $\lambda_0$ to construct a new example
of a quasi-local mass functional.  

To motivate this definition, we will compute the number $\lambda_0$ for concentric
round spheres $\Sigma_r$ in the Schwarzschild manifold of mass $m$, with induced metric $\gamma_r$ and mean 
curvature $H_r$.  For our purposes the Schwarzschild manifold of mass $m$ is $\R^3$ minus the open Euclidean
ball of radius $m/2$, where $m >0$, equipped with the metric
\begin{equation}
\label{eqn_schwarz}
g = \left(1+ \frac{m}{2r}\right)^4 \delta,
\end{equation}
where $\delta$ is the Euclidean metric.  Note that $g$ is scalar-flat and its boundary is a minimal 2-sphere,
called the \emph{horizon}.

Straightforward computations show that $(\Sigma_r, \gamma_r)$ is a round sphere
of area $4\pi r^2\left(1+\frac{m}{2r}\right)^4$, and
$$H_r = \frac{2}{r} \left(1+\frac{m}{2r}\right)^{-2} - \frac{2m}{r^2} \left(1+\frac{m}{2r}\right)^{-3},$$
having used (\ref{eq_conf_mean_curv}).  
The mean curvature $H_r^0$ of $(\Sigma_r, \gamma_r)$ embedded in $\R^3$ is
$$H_r^0 = \frac{2}{r} \left(1+\frac{m}{2r}\right)^{-2}.$$
Therefore, if we let $\lambda_r = H_r^0/H_r$, then $(\Sigma_r, \gamma_r, \lambda_r H_r)$ admits a valid
fill-in -- namely a closed ball in flat-space of boundary area $4\pi r^2\left(1+\frac{m}{2r}\right)^4$.
On the other hand, if $\lambda$ belongs to the interval of positivity for $(\Sigma_r, \gamma_r, H_r)$,
then by Shi--Tam
$$\lambda \leq \frac{\int_\Sigma H_r^0 dA_\gamma}{\int_\Sigma H_r dA_\gamma} = \lambda_r.$$
Thus, $\lambda_r$ is the critical parameter for the Bartnik data.  Some simplifications show
\begin{equation}
\label{eqn_lambda_r}
\lambda_r = \frac{1+\frac{m}{2r}}{1-\frac{m}{2r}}.
\end{equation}
In particular, we have the identity in Schwarzschild space:
$$m = \sqrt{\frac{|\Sigma_r|_g}{16\pi}} \left(1-\frac{1}{\lambda_r^2}\right),$$
for all values of $r$, motivating the following definition of quasi-local mass.
\begin{definition}
\label{def_ql_mass}
Let $\Bcal=(\Sigma, \gamma, H)$ be Bartnik data with critical parameter $\lambda_0$ (from Theorem \ref{thm_interval}).  
Define
$$m(\Bcal)=m(\Sigma, \gamma, H) = \sqrt{\frac{|\Sigma|_\gamma}{16\pi}}\left(1-\frac{1}{\lambda_0^2}\right).$$
\end{definition}
Recall that we assume $\gamma$ has positive Gauss curvature and $H>0$.

\begin{thm} Definition \ref{def_ql_mass} of quasi-local mass satisfies the following properties:
\begin{enumerate}
 \item (nonnegativity) If Bartnik data $\Bcal$ admits a valid fill-in, then its mass $m(\Bcal)$ is nonnegative and is zero
	only if every valid fill-in is static vacuum.  
 \item (spherical symmetry) If Bartnik data $\Bcal$ arises from a coordinate sphere in a Schwarzschild metric of mass 
$m$, then $m(\Bcal)=m$.
 \item (black hole limit).  If $\Bcal_n=(\Sigma, \gamma, H_n)$ is a sequence of Bartnik data and
$H_n \to 0$ uniformly, then
$$\lim_{n \to \infty} m(\Bcal_n) = \sqrt{\frac{|\Sigma|_{\gamma}}{16\pi}}.$$

 \item (ADM-sub-limit)  If $(M,g)$ is an asymptotically flat manifold with nonnegative scalar curvature,
and if $S_r$ is a coordinate sphere of radius $r$ with induced metric $\gamma_r$ and mean curvature $H_r$, then
	\begin{equation}
	\label{eqn_adm_monotonicity}
	 m_{ADM}(M,g) \geq \limsup_{r \to \infty} m(S_r, \gamma_r, H_r).
	\end{equation}
\end{enumerate}
\label{thm_ql_mass}
\end{thm}

\begin{remark}
The proof of Theorem \ref{thm_ql_mass} uses the positive mass theorem \cite{schoen_yau} implicitly, via Lemma \ref{lemma_BY} below,
which relies on the theorem of Shi--Tam.  On the other hand Theorem \ref{thm_ql_mass} also recovers 
the positive mass theorem: if $(M,g)$ is asymptotically flat, has nonnegative scalar curvature, with $\partial M$
empty or consisting of minimal surfaces, then by property (1),
$m(S_r) \geq 0$ for all $S_r$.  From this, inequality (\ref{eqn_adm_monotonicity}) gives $m_{ADM} \geq 0$. 
\end{remark}

\begin{proof}
\emph{Nonnegativity:  }
Observe the following four statements are equivalent, using Theorem \ref{thm_interval}:
$m(\Bcal) >0$; $\lambda_0 > 1$; the number $1$ belongs to the interval
of positivity $I_+$; $\Bcal$ is of positive type.
Also, if $(\Sigma, \gamma, H)$ is of zero type, then $\lambda_0 = 1$ (as follows from Theorem \ref{thm_interval}),
so $m(\Bcal)$ vanishes.  On the other hand, if $m(\Bcal)$ vanishes, then $\lambda_0 = 1$,
so the data is either negative or zero (again, by Theorem \ref{thm_interval}).  But if it is given that the 
data admits a fill-in, then the data must be of zero type.  By Proposition \ref{prop_static}, any such 
fill-in is static vacuum. 

\emph{Spherical symmetry:  } This is clear from the construction at the beginning of this section; we defined 
quasi-local mass so that it has this property.

\emph{Black hole limit:  } It is straightforward to check that if $H_n \to 0$ uniformly, 
then the sequence of critical parameters $\lambda_n$ diverges to infinity.

\emph{ADM-sub-limit:  } For all $r$ sufficiently large, the coordinate spheres $S_r$ have positive mean and Gauss
curvatures.  To prove (\ref{eqn_adm_monotonicity}), recall that the Brown--York mass 
limits to the ADM mass in
the sense that
$$m_{ADM}(M,g) = \lim_{r \to \infty} m_{BY}(S_r).$$
(See Theorem 1.1 of \cite{fan_shi_tam} and the references therein.)
Since we assume $(M,g)$ has nonnegative scalar curvature, $S_r$ is of positive or zero type for all $r$
for which the coordinate sphere is defined.  We invoke Lemma \ref{lemma_BY} below, which states
$m(S_r) \leq m_{BY}(S_r)$, completing the proof.
\end{proof}

\begin{lemma}  
\label{lemma_BY}
For Bartnik data $\Bcal=(\Sigma, \gamma,H)$ of nonnegative type, 
$$m(\Bcal) \leq m_{BY}(\Bcal).$$
\end{lemma}
\begin{proof}
Let $H_0$ be the mean curvature of an isometric embedding of $(\Sigma,\gamma)$ in $\R^3$, which is well-defined
because $K_\gamma >0$.  By Shi--Tam,
we have $\lambda_0 \leq \frac{\int_\Sigma H_0 dA_\gamma}{\int_\Sigma H dA_\gamma}$.
In particular,
\begin{align}
 m(\Bcal) &= \sqrt{\frac{|\Sigma|_\gamma}{16\pi}}\left(1-\frac{1}{\lambda_0^2}\right)\nonumber\\
	&\leq \sqrt{\frac{|\Sigma|_\gamma}{16\pi}}\left(1-\left(\frac{\int_\Sigma H dA_\gamma}{\int_\Sigma H_0 dA_\gamma}\right)^2\right) \label{miao_mass}\\
	&\leq \sqrt{\frac{|\Sigma|_\gamma}{16\pi}}\left(\frac{\left(\int_\Sigma H_0 dA_\gamma+ \int_\Sigma HdA_\gamma\right)\left(\int_\Sigma H_0dA_\gamma - \int_\Sigma HdA_\gamma\right)}{\left(\int_\Sigma H_0dA_\gamma\right)^2 }\right)\nonumber\\
	&\leq \sqrt{\frac{|\Sigma|_\gamma}{16\pi}}\frac{16\pi m_{BY}(\Bcal)}{\int_\Sigma H_0dA_\gamma }\nonumber,
\end{align}
where again we have used Shi--Tam and the fact that the data is of nonnegative type.
The Minkowski inequality for convex regions in $\R^3$ \cite{isoperimetric}
states that
$$\left(\int_\Sigma H_0 dA_\gamma\right)^2 \geq 16\pi |\Sigma|_\gamma.$$
Together with the above, this completes the proof.
\end{proof}
The right-hand side of (\ref{miao_mass}) is a definition of quasi-local mass proposed by 
Miao, which he observed is bounded above by the Brown--York mass using the same argument \cite{miao_local_rpi}.

\subsection{Physical remarks}
\label{sec_physical}
It has been suggested in the literature (see \cite{bartnik_3_metrics} for instance) that if the quasi-local
mass of the boundary of a region $\Omega$ vanishes, then $\Omega$ ought to be flat.  The Brown--York mass
and Bartnik mass both satisfy this property (see \cites{bartnik_mass, imcf}).  Definition \ref{def_ql_mass}
suggests an alternative viewpoint that such $\Omega$ ought to be \emph{static vacuum}, which includes flat
metrics as a special case.  Indeed, one could make a physical argument that in a region of a spacetime
that is static vacuum, quasi-local mass should vanish since there is no matter content and no gravitational
dynamics (cf. \cite{anderson}, which also discusses the vanishing of quasi-local mass on static vacuum
regions).

\section{Examples}
\label{sec_examples}
Let $(M,g)$ be a Riemannian 3-manifold.  If $\Omega$ is a subset of $M$ with boundary $\partial \Omega$ 
homeomorphic to $S^2$, and if $\partial \Omega$ has positive mean curvature $H$ (with respect to some
chosen normal direction),
define
$$m(\Omega)=m(\partial \Omega, g_{T\partial \Omega}, H),$$
where $T\partial \Omega$ is the tangent bundle of $\partial \Omega$.  If $g$ has nonnegative scalar curvature,
then $m(\Omega) \geq 0$ by Theorem \ref{thm_ql_mass}.

\emph{Euclidean space:  }
Consider $\R^3$ with the standard flat metric.  Let $\Omega \subset \R^3$ be a strictly convex open set with 
smooth boundary that is not round, with mean curvature $H_0$ and induced metric $\gamma_0$.  There is no
valid fill-in of $(\partial \Omega,\gamma_0, H_0)$ with mean curvature $H > H_0$; this statement follows
from the Shi--Tam inequality (\ref{eqn_shi_tam}) or alternatively by Miao's ``positive mass theorem with 
corners'' \cite{miao}.  This implies $\lambda_0=1$, and so
$m(\Omega)=0$.  The Brown--York mass of $\Omega$ also vanishes, as $H_0=H$.  
A straightforward computation shows that the Hawking mass of $\Omega$ is 
strictly negative.

\emph{Schwarzschild, positive mass:  }
Next let $(M,g)$ be a Schwarzschild manifold of mass $m>0$ (see equation (\ref{eqn_schwarz})).  Suppose 
$\Omega \subset M$ is topologically an open 3-ball with boundary $\Sigma$ disjoint from the horizon.  
\begin{lemma}
For the Bartnik data induced on $\partial \Omega$, $\lambda_0 = 1$.  Equivalently, $m(\Omega)=0$.
\end{lemma}
Figure \ref{fig_schwarz} gives a depiction of the Bartnik data in question.
\begin{figure}[ht]
\caption{Off-center ball in Schwarzschild}
\begin{center}
\includegraphics[scale=0.7]{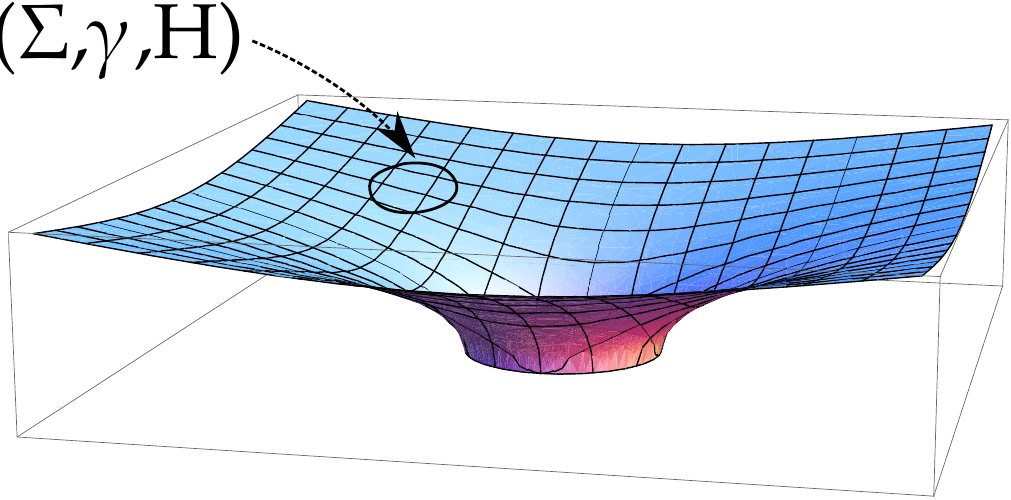}
\end{center}
\flushleft\footnotesize{The Bartnik data $(\Sigma, \gamma, H)$ arises from the boundary of a small
ball away from the horizon in a Schwarzschild manifold.}
\label{fig_schwarz}
\end{figure}

\begin{proof}
Certainly $\lambda_0 \geq 1$,
since $\Omega$ is tautologically a valid fill-in.  If $\lambda_0> 1$, there exists a valid fill-in 
$\Omega'$ of $\Sigma$ (with the same boundary metric
and mean curvature), such that $\partial \Omega' \sm \Sigma$ is nonempty and consists of minimal surfaces.  Glue
$\Omega'$ to $M \sm \Omega$  along $\Sigma$, obtaining a manifold $(M', g')$ that is smooth and has
nonnegative scalar curvature away from $\Sigma$.  Moreover, $g'$ is Lipschitz across $S$, and
$\partial M'$ consists of minimal surfaces (including the Schwarzschild horizon).  Let $A$ and $A'$
be the minimum areas in the homology class of the boundary for the respective manifolds $(M,g)$ and $(M',g')$.
$A$ is attained uniquely by the horizon $S$ in $M$, and by a similar consideration $A'$ is attained by
a surface $S'$ that includes $S$ as a proper subset. (To see this, observe that the Schwarzschild manifold minus its horizon is foliated by
the $\{r= \text{const.}\}$ spheres, which are convex; thus $S'$ may not intersect the interior of $M' \sm \Omega$ yet must intersect $M' \sm \Omega$ to be homologous
to the boundary of $M'$.) Thus $A' > A$.
By direct computation,
$$m = \sqrt{\frac{A}{16\pi}},$$
and so
\begin{equation}
\label{eqn_rpi_violated}
m' < \sqrt{\frac{A'}{16\pi}},
\end{equation}
where $m'=m$ is the ADM mass of $(M',g')$ (equal because $g$ and $g'$ agree outside a compact set).  Using
an argument similar to Miao \cite{miao}, one can mollify $(M', g')$ to a smooth, asymptotically flat metric
of nonnegative scalar curvature and minimal boundary that gives strict inequality in (\ref{eqn_rpi_violated}).
This violates the well-known Riemannian Penrose inequality \cites{imcf,bray_RPI}.  
This contradiction implies that $\lambda_0=1$, so $m(\Omega)=0$.
\end{proof}

Thus, we have examples of Bartnik data of zero type that do not arise as the boundaries of regions in flat space.  
In other words, we have non-flat domains $\Omega$ for which $m(\Omega)=0$.  Of course, by Theorem
\ref{thm_ql_mass}, such $\Omega$ must be static vacuum (as is the case for the Schwarzschild metric).

To further extend this example, Huisken and Ilmanen \cite{imcf} show that there exist small balls $\Omega$ away from
the horizon in the Schwarzschild manifold whose Bartnik data $(\Sigma, \gamma, H)$
have strictly positive Hawking mass.
Moreover, the case of equality of Theorem \ref{thm_shi_tam} of Shi and Tam shows that the Brown--York 
mass of $\Omega$ is also strictly positive.  It follows that  for
$\lambda> 1$ sufficiently close to 1, the data $(\Sigma, \gamma, \lambda H)$
is of negative type, yet still has strictly positive $m_H$ and $m_{BY}$.

Note that we have not stated $\Omega$ being static vacuum implies $m(\Omega)=0$.  Counterexamples are unknown
to the author.

\emph{Schwarzschild, negative mass:  }
Let $(M,g)$ be the Schwarzschild metric of mass $m<0$ (defined by (\ref{eqn_schwarz}) on $\R^3$ minus
the closed ball of radius $|m|/2$).  The Bartnik data induced on spheres $\{r=\text{const.}\}$
is of negative type because the critical parameter $\lambda_r$ is less than one by (\ref{eqn_lambda_r}).

\subsection{Example $m(\lambda)$ function}
\label{sec_example_m_lambda}
Here we give an explicit computation of the inner mass function $m(\lambda)$ defined in section \ref{sec_inner_mass_function}
for Bartnik data $\Bcal$ corresponding to the coordinate sphere $S_r$ of radius $r>m/2$ in the Schwarzschild metric of mass $m>0$, with induced
metric $\gamma$ and mean curvature $H$. The Riemannian Penrose inequality \cite{bray_RPI}, 
shows that\footnote{In a Schwarzschild manifold, $m=\sqrt{\frac{A}{16\pi}}$,
where $A$ is the area of the horizon.  Now, $m_{inner}(\Bcal) \geq m$ follows from the definition.  If $m_{inner}(\Bcal) > m$, there exists
a fill-in of $\Bcal$ with minimum area $A'>A$ attained by a minimal surface.  One can then arrange a strict violation of the Penrose
inequality by gluing the exterior Schwarzschild region of $\Bcal$ to the fill-in.  The gluing is only Lipschitz across $\Bcal$, but
the smoothing and conformal techniques in Miao \cite{miao} can be used to produce a smooth example, leading to a contradiction.} the Bartnik inner mass of $\Bcal$ equals $m$.  Let $\lambda > 0$;
the data $(S_r, \gamma, \lambda H)$ embeds uniquely as a coordinate sphere $S_{r'}$ of some radius $r'$ in a Schwarzschild metric of some mass $m'$.  
Equating the areas of $S_r$ and $S_{r'}$ in their respective metrics, we have
\begin{equation}
4\pi r^2 \left(1+ \frac{m}{2r}\right)^4 = 4\pi (r')^2 \left(1+ \frac{m'}{2r'}\right)^4.
\end{equation}
Equating $\lambda H$ with the mean curvature of $S_{r'}$ leads to
\begin{equation}
\lambda \left(\frac{2}{r}\left(1+\frac{m}{2r}\right)^{-2} - \frac{2m}{r^2}\left(1+\frac{m}{2r}\right)^{-3} \right) 
=\frac{2}{r'}\left(1+\frac{m'}{2r'}\right)^{-2} - \frac{2m'}{(r')^2}\left(1+\frac{m'}{2r'}\right)^{-3}.
\end{equation}
With some calculations, one can compute $r'$ and $m'$ explicitly.  For $\lambda \in (0, \lambda_0)$, we know
$m(\lambda)$, the Bartnik inner mass of $S_{r'}$, simply equals $m'$ (again, by the Riemannian Penrose inequality).
Omitting some details, we give the formula:
\begin{equation}
m(\lambda) = \frac{r}{2} \left(\left(1+\frac{m}{2r}\right)^2-\lambda^2\left(1-\frac{m}{2r}\right)^2\right).
\end{equation}
As anticipated by Proposition \ref{prop_m_lambda}, $m(\lambda)$ is continuous, decreasing, 
and $m(0) = \sqrt{\frac{A}{16\pi}}$, where $A$ is the area of $S_r$ in the Schwarzschild metric of mass $m$.  Moreover, $m(\lambda)$ vanishes at the critical value $\lambda_0 = \frac{1+\frac{m}{2r}}{1-\frac{m}{2r}}$ (computed in section \ref{sec_ql_mass}), a property conjectured to hold in general (see the paragraph following Problem \ref{prob_zero_type} 
in section \ref{sec_open_problems}).

\section{An algebraic operation on quasi-local mass functionals}
\label{sec_product}
For a quasi-local mass functional $m_i$ (i.e., a map from the set of Bartnik data to the real numbers), define
the following quantity in $[-\infty, \infty]$:
$$\lambda_i(\Sigma, \gamma, H) = \sup \{ \lambda>0 : 
	m_i(\Sigma, \gamma, \lambda H) \geq 0\}.$$
In other words, $\lambda_i$ measures how much one can scale the boundary mean curvature until the mass
$m_i$ becomes negative.  Up to this point, we have studied this quantity for the case in which
$m_i$ is the Bartnik inner mass (since $m_{inner}(\Sigma, \gamma, H) \geq 0$ if and only if
$(\Sigma, \gamma, H)$ has a valid fill-in).  Theorem \ref{thm_interval} implies that $\lambda_i$ is a 
positive, finite number for the case $m_i =m_{inner}$.

Here we use the number $\lambda_i$ to construct an algebraic product of two quasi-local
mass functionals, of which that constructed in section \ref{sec_ql_mass} is a special case.
We restrict to quasi-local mass functionals $m_i$ satisfying the following mild assumptions
on all Bartnik data:
\begin{enumerate}
 \item $\lambda_i(\Sigma, \gamma, H)$ is a positive real number, and
 \item $m_i(\Sigma, \gamma, \lambda H)$ is decreasing as a function of $\lambda$.
\end{enumerate}
For example the Hawking mass, Brown--York
mass, and Bartnik inner mass (see Proposition \ref{prop_m_lambda}) satisfy these properties.

Define the following binary operation on the set of quasi-local mass functionals. 
Given $m_1$ and $m_2$, let
\begin{equation}
 (m_1 * m_2) (\Sigma, \gamma, H) = m_1 \left(\Sigma, \gamma, \frac{\lambda_1}{\lambda_2} H\right),
\end{equation}
where $\lambda_i=\lambda_i(\Sigma, \gamma, H)$ for $i=1,2$.  This operation satisfies a number of properties.
\begin{prop} Let $m_1, m_2,$ and $m_3$ be quasi-local mass functionals.
\begin{enumerate}
 \item $m_1 * m_1 = m_1$.
 \item $(m_1 * m_2) * m_3 = m_1 * m_3 = m_1 * (m_2 * m_3)$.  In particular, $*$ is associative.
 \item $m_2$ controls the sign of $m_1 * m_2$ in the following sense:
  \begin{enumerate}
   \item $m_1 * m_2 (\Sigma, \gamma, H) > 0$ if and only if $m_2(\Sigma, \gamma, H) > 0$, and
   \item $m_1 * m_2 (\Sigma, \gamma, H) = 0$ if and only if $m_2(\Sigma, \gamma, H) = 0$.
  \end{enumerate}
  \item If $m_1$ has the black hole limit property (see Theorem \ref{thm_ql_mass}), so does $m_1 * m_2$.
  \item If both $m_1$ and $m_2$ produce the value $m$ on concentric round spheres in the Schwarzschild metric
  of mass $m$, then so does $m_1 * m_2$.
  \item If $m_2 \leq m_3$ (as functions), then $m_1 * m_2 \leq m_1 * m_3$.
  \end{enumerate}
\end{prop}
The above properties all essentially follow immediately from the definitions and so we omit the proof.  Below we sketch some
of the steps as a sample.
\begin{proof}[Sketch:]
We compute $(m_1 * m_2) * m_3$.  First, $(m_1 * m_2) (\Sigma, \gamma, H) = m_1(\Sigma, \gamma, \frac{\lambda_1}{\lambda_2} H)$
has critical parameter $\lambda_2$.  Then
\begin{align*}
((m_1 * m_2) * m_3(\Sigma, \gamma, H)) &= (m_1 * m_2)\left(\Sigma, \gamma, \frac{\lambda_2}{\lambda_3} H\right)\\
&=m_1\left(\Sigma, \gamma, \frac{\lambda_1}{\lambda_2} \frac{\lambda_2}{\lambda_3}H\right),
\end{align*}
which equals $(m_1 * m_3)(\Sigma, \gamma, H)$.

We also demonstrate property (3a).  Note $(m_1 * m_2) (\Sigma, \gamma, H) = m_1(\Sigma, \gamma, \frac{\lambda_1}{\lambda_2} H)$
is positive if and only if $\frac{\lambda_1}{\lambda_2} < \lambda_1$; that is, $\lambda_2 > 1$.  However, $\lambda_2 > 1$ if and only
if $m_2(\Sigma, \gamma, H) > 0$.
\end{proof}

In the next section, we demonstrate $m_1 * m_2$ generally does not equal $m_2 * m_1$.

\subsection{Examples of $m_1 * m_2$} $\;$

\emph{Hawking mass and Bartnik inner mass:  } 
The quasi-local mass of Definition \ref{def_ql_mass} is equal to $m_H * m_{inner}$, where
$m_H$ is the Hawking mass.  To see this, note that
$\lambda_H = \sqrt{\frac{16\pi}{\int_\Sigma H^2 dA_\gamma}}$ and
$\lambda_{inner} = \lambda_0$. 
Then by definition,
\begin{align*}
m_H * m_{inner} (\Sigma, \gamma, H) &= m_H\left(\Sigma, \gamma, \frac{\lambda_H}{\lambda_0} H\right)\\
	&= \sqrt{\frac{|\Sigma|_\gamma}{16\pi}}\left(1 - \frac{1}{\lambda_0^2}\right).
\end{align*}
We reiterate that $m_H * m_{inner}$ inherits the following property
from $m_{inner}$: vanishing precisely on Bartnik data of zero type.

\emph{Hawking mass and Brown--York mass:  }
To compute $m_H * m_{BY}$, we note $\lambda_H$ was found in the last example, and
$\lambda_{BY} = \frac{\int_{\Sigma} H_0 dA_\gamma}{\int_\Sigma H dA_\gamma}.$
Using the definition,
$$m_H * m_{BY} (\Sigma, \gamma, H) = \sqrt{\frac{|\Sigma|_\gamma}{16\pi}}\left(1- 
  \left(\frac{\int_\Sigma H dA_{\gamma}}{\int_\Sigma H_0 dA_{\gamma}}\right)^2\right).$$
This quasi-local mass was written down in a different context by Miao \cite{miao_local_rpi}.

\emph{Brown--York mass and Hawking mass:  }
The steps from the last example show
$$m_{BY}*m_H (\Sigma, \gamma, H) = \int_\Sigma H_0 dA_\gamma \left(1- \sqrt{\frac{\int_\Sigma H^2 dA_\gamma}{16\pi}}\right),$$
illustrating concretely the non-commutativity of $*$.

\section{Concluding remarks and open problems}
\label{sec_open_problems}
We conclude by mentioning some questions raised in this paper.
\begin{prob}
Determine whether the quasi-local mass of Definition \ref{def_ql_mass} is monotone under some flow.
\end{prob}
Monotonicity means that if $\{(\Sigma_t,\gamma_t, H_t)\}_{t \in [0, \epsilon)}$ is some family of 
surfaces (together with their Bartnik data) moving outward in a manifold of nonnegative scalar curvature, then
$m(\Sigma_t, \gamma_t, H_t)$ is non-decreasing.  Monotonicity is often (but not universally) suggested as a 
desirable property of quasi-local mass \cite{bartnik_3_metrics}.

\begin{prob}
\label{prob_zero_type}
Determine whether the Bartnik data $(\Sigma, \gamma, \lambda_0H)$ is of zero type.  Equivalently, construct 
a static vacuum fill-in of $(\Sigma, \gamma, \lambda_0H)$.
\end{prob}
That the two above statements are equivalent follows from Proposition \ref{prop_static} and 
Theorem \ref{thm_interval}.  The precise nature of the Bartnik data rescaled with the critical parameter
$\lambda_0$ is perhaps the biggest open question of this paper.  
An affirmative answer to Problem \ref{prob_zero_type} would imply that 
$(\Sigma, \gamma, \lambda_0 H)$ admits a static vacuum fill-in. In general, constructing
static vacuum metrics with prescribed boundary data is a very difficult problem (cf. the
work of Anderson and Khuri on static vacuum asymptotically flat ``extensions'' of 
Bartnik data \cite{ak}).  

More generally, one could ask what happens
to the geometry of the class valid fill-ins of $(\Sigma, \gamma, \lambda H)$ in the limit
$\lambda \nearrow \lambda_0$.  An optimistic conjecture would be that in the limit 
$\lambda \nearrow \lambda_0$, any valid fill-in $(\Omega_\lambda, g_\lambda)$
of $(\Sigma, \gamma, \lambda H)$ satisfies:
\begin{itemize}
 \item the black holes (area-minimizing minimal surfaces) in $(\Omega_\lambda, g_\lambda)$ are shrinking to zero size
 (i.e., $\lim_{\lambda \to \lambda_0^-} m(\lambda)=0$), and
 \item the metric $g_\lambda$ is approaching a static vacuum metric in an appropriate sense.
\end{itemize}
There may be a connection between the first point and Miao's localized Riemannian Penrose inequality 
\cite{miao_local_rpi}.

The above discussion is basically a localization of the near-equality case of the positive
mass theorem \cite{schoen_yau}.  In such a 
global setting, the question is: what happens to the geometry of a sequence of asymptotically flat 
manifolds $(M_i, g_i)$ of nonnegative scalar curvature whose total mass is approaching zero?  The Riemannian Penrose 
inequality \cites{imcf, bray_RPI} shows that any black holes in $(M_i,g_i)$ must be approaching zero, and some
partial results exist for proving that $g_i$ is approaching a flat metric \cites{bartnik_tsing, dan_lee, bray_finster, lee_sormani}.

\appendix
\section{Conformal transformation of curvatures}
\label{appendix_formulas}
We repeatedly used the following formulas that relate the scalar curvature and mean curvature of conformal metrics.
Suppose $g$ and $\ol g$ are Riemannian metrics on a 3-manifold for which $\ol g = u^4 g$ for some smooth 
function $u>0$.  If $R$ and $\ol R$ are the scalar curvatures of $g$ and $\ol g$, then
\begin{equation}
\ol R = u^{-5}(-8\Delta u + R u),
\label{eq_conf_scalar_curv}
\end{equation}
where $\Delta$ is the Laplacian with respect to $g$.  
Next, suppose $S$ is a hypersurface with unit normal field $\nu$ with respect to $g$.  Then
the mean curvatures $H$ and $\ol H$ (in the direction defined by $\nu$) with respect to
$g$ and $\ol g$ satisfy:
\begin{equation}
\ol H = u^{-2} H + 4u^{-3} \nu(u).
\label{eq_conf_mean_curv}
\end{equation}

\section{Geometric measure theory}
\label{appendix_gmt}
Here is an extremely useful result from geometric measure theory on the existence and regularity
of area-minimizing surfaces.
\begin{thm}
\label{gmt_lemma}
Let $(M,g)$ be a smooth, compact Riemannian manifold of dimension \mbox{$2 \leq n \leq 7$} with boundary $\partial M$.  
Suppose $\partial M$ has positive mean curvature (i.e., inward-pointing mean curvature vector).  Given a connected component 
$S$ of $\partial M$, there exists a smooth, embedded hypersurface $\tilde S$ of zero mean curvature that minimizes
area among surfaces homologous to $S$.  Moreover, $\tilde S$ does not intersect $\partial M$.
\end{thm}
These results are essentially due to Federer and Fleming \cites{fed_flem, fleming, federer}.  The rough idea of the proof of Theorem \ref{gmt_lemma} is to take a minimizing sequence of surfaces $\{S_i\}$ (viewed as integral currents) in $[S]$, the homology class of $S$.  By the Federer--Fleming compactness theorem, some subsequence converges to a surface $\tilde S$.  Standard arguments show that $\tilde S$ remains in $[S]$ and indeed has the desired minimum of area.  Regularity theory (requiring $n \leq 7$) proves that $\tilde S$ is a smooth, embedded hypersurface.  By the first variation of area formula, $\tilde S$ has zero mean curvature and may not touch the positive mean curvature boundary (which acts as a barrier).  See the appendix
of \cite{schoen_yau} for a careful proof of the last fact.

\section{Deformations of scalar curvature near a boundary}
\label{appendix_bmn}
Here we prove the following useful lemma.
\begin{lemma}
\label{lemma_positive}
Suppose that $(\Sigma, \gamma, H_1)$ admits a valid fill-in.  If $0 < H_2 < H_1$, then
$(\Sigma, \gamma, H_2)$ admits a valid fill-in with positive scalar curvature at a point.  In particular, $(\Sigma, \gamma, H_2)$
is of positive type.
\end{lemma}
Although we only prove the case $K_\gamma > 0$ here, Lemma \ref{lemma_positive} is true without this hypothesis.
The proof is an application of techniques developed recently by Brendle, Marques, and Neves \cite{bmn}.  
\begin{proof}
\emph{Step 1:  } We construct a valid fill-in of $(\Sigma, \gamma)$ with mean curvature strictly greater than 
$H_2$ and with positive scalar curvature in a neighborhood of $\Sigma$.  

Since $\Sigma$ is compact, we may choose $\alpha \in (0,1)$ so that $\alpha H_1 > H_2$.  We proved in 
step 2 of Theorem \ref{thm_interval} that $(\Sigma, \gamma, \alpha H_1)$ is of positive type
and moreover admits a valid fill-in
$(\Omega, g_1)$ whose scalar curvature is strictly positive in a neighborhood $U$ of $\Sigma$. (For the latter statement,
refer to equation (\ref{eqn_R_tilde_g}) and note that $\rho'(0)>0$.)  

\emph{Step 2:  } We define a metric $g_2$ on $\Omega$ as follows, with the goal of making the boundary mean
curvature of $g_2$ equal to $H_2$.  First, consider a neighborhood of $\Sigma$ contained in $U$ that is diffeomorphic to
 $\Sigma \times (-t_0, 0]$
(where $t=0$ corresponds to $\Sigma$).  Define for $x \in \Sigma$ and $t \in (-t_0,0]$:
$$g_2(x,t) = \rho(t)^2dt^2 +(1+tH_2(x)) \gamma(x),$$
where $\rho(t)$ is a function satisfying $\rho(0)=1$ and will be specified later.  It is readily checked
that $g_2$ induces on $\Sigma$ the metric $\gamma$ and mean curvature $H_2$.  Shrinking $t_0$ if necessary
and choosing $\rho(t)$ bounded below by a positive constant with $\rho'(t)>0$ sufficiently large, 
we may arrange $g_2$ to have strictly positive scalar curvature on $\Sigma \times (-t_0,0]$.  This is readily checked using
equation (\ref{eq_second_variation}).  
Now, extend $g_2$ arbitrarily to a smooth
metric on $\Omega$ (not necessarily preserving nonnegative scalar curvature).  Replace $U$ with the smaller
neighborhood $\Sigma \times (-t_0,0]$

To summarize, we have two metrics $g_1$ and $g_2$ on the compact manifold $\Omega$, inducing
boundary data $(\Sigma, \gamma, \alpha H_1)$ and $(\Sigma, \gamma, H_2)$, respectively, each with positive scalar
curvature on the neighborhood $U$ of $\Sigma$.  By compactness, the scalar curvatures of $g_1|_U$ and $g_2|_U$
are bounded below by a constant $R_0>0$.

\emph{Step 3:  } Apply Theorem 5 of Brendle--Marques--Neves \cite{bmn} to produce a metric $\hat g$ on $\Omega$ 
satisfying the following properties\footnote{Due to the local nature of the construction, it is clear that we can ignore 
any connected components of $\partial\Omega$ that are not $\Sigma$.}:
\begin{enumerate}
 \item[(i)] $R_{\hat g}(x) \geq \min \{R_{g_1}(x), R_{g_2}(x) \} - \frac{R_0}{2}$,
 \item[(ii)] $\hat g$ agrees with $g_1$ outside of $U$,
 \item[(iii)] $\hat g$ agrees with $g_2$ in some neighborhood of $\Sigma$.
\end{enumerate}
(To apply the theorem, it is crucial that $\alpha H_1 > H_2$.) 

By the third condition, $(\Omega, \hat g)$ is a fill-in of $(\Sigma, \gamma, H_2)$.  By the first and second
conditions, $\hat g$ has nonnegative (but not identically zero) scalar curvature and $\partial \Omega \sm \Sigma$ (if nonempty) is a minimal
surface.  In particular, $(\Omega, \hat g)$ is a valid fill-in with positive scalar curvature at some point.

Finally, the last statement in the lemma follows from Proposition \ref{prop_positive}.
\end{proof}

\end{document}